\documentclass{article}
\usepackage[utf8]{inputenc}

\usepackage{amsmath,amssymb,amsfonts,amsthm}

\usepackage{hyperref}
\hypersetup{
    pdfnewwindow=true,     
    colorlinks=true,     
    linkcolor=blue,     
    citecolor=blue,     
    filecolor=magenta,     
    urlcolor=cyan     
}

\usepackage[capitalize]{cleveref}
\usepackage{enumitem}
\usepackage{cancel}
\usepackage{mathtools}
\usepackage{bbm,dsfont}
\usepackage{nicefrac}

\usepackage{natbib}
\newtheorem{theorem}{Theorem}

\newtheorem{lemma}[theorem]{Lemma}

\newtheorem{claim}[theorem]{Claim}

\usepackage{xspace}
\usepackage[colorinlistoftodos, color=blue!30!white, 
]{todonotes}                                        

\newcommand{\KL}{\operatorname{KL}} 
\newcommand{\rep}[1]{\ensuremath ^{(#1)}} 

\renewcommand{\P}{\operatorname{\mathbb{P}}} 
\newcommand{\E}{\operatorname{\mathbb{E}}} 

\newcommand{\R}{\mathbb{R}}  

\newcommand*\diff{\mathop{}\!\mathrm{d}}

\newcommand*{\quotientspace}[2]
{\ensuremath{\raisebox{.25ex}{\ensuremath{#1}} \hspace{-.35ex} \raisebox{-.25ex}{/} 
\hspace{-.05ex} \raisebox{-.85ex}{\ensuremath{#2}}} \hspace{.2ex} }
\newcommand*{\textquotientspace}[2]
{\ensuremath{\raisebox{.05ex}{\ensuremath{#1}} \hspace{-.30ex} \raisebox{-.15ex}{/} 
\hspace{.10ex} \raisebox{-.55ex}{\ensuremath{#2}}}}

\newcommand{\upsum}
{\ensuremath \mkern5.4mu \overline{\vphantom{S}\mkern8mu} \mkern-11mu S}
\newcommand{\lowsum}
{\ensuremath \mkern3.4mu \underline{\vphantom{S}\mkern8mu} \mkern-9mu S}
\newcommand{\upint}
{\ensuremath \mkern10.4mu \overline{\vphantom{\int}\mkern7mu} \mkern-21mu\int}
\newcommand{\upintwdomain}[1]
{\ensuremath \mkern10.4mu \overline{\vphantom{\int}\mkern7mu} \mkern-21mu\int_{#1}}
\newcommand{\textupint}
{\textstyle \mkern5mu \overline{\vphantom{\int}\mkern6mu} \mkern-14mu\intop\mkern0mu}
\newcommand{\textupintwdomain}[1]
{\textstyle \mkern5mu \overline{\vphantom{\int}\mkern6mu} \mkern-14mu\intop_{#1}}
\newcommand{\lowint}
{\ensuremath \underline{\vphantom{\int}\mkern7mu} \mkern-9mu\int}
\newcommand{\lowintwdomain}[1]
{\ensuremath \underline{\vphantom{\int}\mkern7mu} \mkern-9mu\int_{#1}}
\newcommand{\textlowint}
{\textstyle \mkern3mu \underline{\vphantom{\int}\mkern6mu} \mkern-8.5mu\intop}
\newcommand{\textlowintwdomain}[1]
{\textstyle \mkern3mu \underline{\vphantom{\int}\mkern6mu} \mkern-8.5mu\intop_{#1}}

\def\ddefloop#1{\ifx\ddefloop#1\else\ddef{#1}\expandafter\ddefloop\fi}

\def\ddef#1{\expandafter\def\csname c#1\endcsname{\ensuremath{\mathcal{#1}}}}
\ddefloop ABCDEFGHIJKLMNOPQRSTUVWXYZ\ddefloop

\def\ddef#1{\expandafter\def\csname bb#1\endcsname{\ensuremath{\mathbb{#1}}}}
\ddefloop ABCDEFGHIJKLMNOPQRSTUVWXYZ\ddefloop

\def\ddef#1{\expandafter\def\csname fr#1\endcsname{\ensuremath{\mathfrak{#1}}}}
\ddefloop ABCDEFGHIJKLMNOPQRSTUVWXYZabcdefghijklmnopqrstuvwxyz\ddefloop

\def\ddef#1{\expandafter\def\csname ss#1\endcsname{\ensuremath{\mathsf{#1}}}}
\ddefloop ABCDEFGHIJKLMNOPQRSTUVWXYZabcdefghijklmnopqrstuvwxyz\ddefloop

\def\ddef#1{\expandafter\def\csname bf#1\endcsname{\ensuremath{\mathbf{#1}}}}
\ddefloop ABCDEFGHIJKLMNOPQRSTUVWXYZabcdefghijklmnopqrstuvwxyz\ddefloop

\def\ddef#1{\expandafter\def\csname v#1\endcsname{\ensuremath{\boldsymbol{#1}}}}
\ddefloop ABCDEFGHIJKLMNOPQRSTUVWXYZabcdefghijklmnopqrstuvwxyz\ddefloop

\def\ddef#1{\expandafter\def\csname v#1\endcsname{\ensuremath{\boldsymbol{\csname #1\endcsname}}}}
\ddefloop {alpha}{beta}{gamma}{delta}{epsilon}{varepsilon}{zeta}{eta}{theta}{vartheta}{iota}{kappa}{lambda}{mu}{nu}{xi}{pi}{varpi}{rho}{varrho}{sigma}{varsigma}{tau}{upsilon}{phi}{varphi}{chi}{psi}{omega}{Gamma}{Delta}{Theta}{Lambda}{Xi}{Pi}{Sigma}{varSigma}{Upsilon}{Phi}{Psi}{Omega}\ddefloop

\newcommand{\VES}{V^{\text{\scshape{es}}}}
\newcommand{\VKS}{V^{\text{\scshape{ks}}}}

\newcommand{\pr}[1]{\left( #1 \right)}
\newcommand{\br}[1]{\left[ #1 \right]}
\newcommand{\cbr}[1]{\left\{ #1 \right\}}

\title{A note on 
a confidence bound of \\
Kuzborskij and Szepesv{\'a}ri}

\author{
   Omar Rivasplata \\ 
  DeepMind \& University College London  
}

\date{December 25, 2020}

\begin{document}

\maketitle

\begin{abstract}
In an interesting recent work, Kuzborskij and Szepesv{\'a}ri
derived a confidence bound for functions of independent random variables, which is based on an inequality that relates concentration to squared perturbations of the chosen function.
Kuzborskij and Szepesv{\'a}ri also established the PAC-Bayes-ification of their confidence bound.
Two important aspects of their work are that the random variables could be of unbounded range, and not necessarily of an identical distribution.
The purpose of this note is to advertise/discuss these interesting results,
with streamlined proofs. 
This expository note is written for persons who, metaphorically speaking, enjoy the `featured movie' but prefer to skip the preview sequence.
\end{abstract}

\section{Introduction}

In an interesting recent work, \citeauthor{kuzborskij2019efron} 
derived a confidence bound for the random variable
\[
  \Delta = f(S) - \E[f(S)]
\]
where $S = (Z_1, Z_2, \ldots, Z_n)$ is a size-$n$ random sample composed of independent $\cZ$-valued random elements $Z_i$, and  $f : \cZ^n \to \R$ is a measurable function.
Notice, however, that the components 
are not required to be identically distributed: each $Z_i$ may be distributed according to a different%
\footnote{$\cM_1(\cA,\Sigma_{\cA})$ denotes the family of probability measures defined on a measurable space $(\cA,\Sigma_{\cA})$. When $\Sigma_{\cA}$ is clear from the context, we write simply $\cM_1(\cA)$ for simplicity.} 
$\mu_i \in \cM_1(\cZ)$.
Accordingly, the distribution of the size-$n$ random sample $S$ is
$P_n = \mu_1 \otimes\cdots\otimes \mu_n$.

Their confidence bound 
is based on an estimator of the variance of $f(S)$.
Recall that McDiarmid's inequality, which is based on the bounded differences property, relates concentration of $\Delta$ around zero (its mean) to 
the sensitivity of $f$ to coordinatewise perturbations (``first-order'').
By contrast, the bound of \citeauthor{kuzborskij2019efron} relates concentration to squared perturbations (``second-order''), which leads to an inequality based on a variance estimator. 
The latter has a resemblance with a well-known estimator, recalled next.

\paragraph{The variance estimator used in the Efron-Stein inequality.}
This is defined as follows:
\begin{equation}
  \label{eq:VES}
  \VES = \sum_{k=1}^n \E\left[(f(S) - f(S\rep{k}))_+^2 \,\middle|\, S \right]\,,
\end{equation}
where $(s)_+ = \max\{0,s\}$ is the positive part, and the notation $S\rep{k}$ indicates that the $k$th element of $S$ is replaced with $Z_k'$, where $S'= (Z_1', Z_2', \ldots, Z_n')$ is an independent copy of $S = (Z_1, Z_2, \ldots, Z_n)$.
Further details about this estimator, with context and references, 
can be found in \citet{boucheron2013book}.

Problem: In order to prove a confidence bound for $\Delta$ based on $\VES$, one needs 
\emph{a priori} assumptions on the moments of $\VES$. 
To avoid this limitation, \citeauthor{kuzborskij2019efron} used a modified variance estimator.

\paragraph{The variance estimator used in the Kuzborskij-Szepesv{\'a}ri inequality.}
This is defined as follows:
\begin{equation}
  \label{eq:VKS}
  \VKS = \sum_{k=1}^n \E\left[(f(S) - f(S\rep{k}))^2 \,\middle|\, Z_1, \ldots, Z_k \right]\,.
\end{equation}
\citeauthor{kuzborskij2019efron} called it a ``semi-empirical'' estimator, because of its dependence on both the sample and the distribution of the sample. 


The main result of \citeauthor{kuzborskij2019efron} is the following high-confidence bound:
For any $y > 0$ and $x \geq 1$,
with probability at least $1-e^{-x}$ one has
\begin{equation}
  \label{eq:intro:es_1}
  |\Delta| \leq \sqrt{2 (\VKS + y) \left[ 1 + \frac{1}{2} \log\pr{1 + \frac{\VKS}{y}} \right] x}\,.
\end{equation}

\paragraph{Remark:}
Inequality~\eqref{eq:intro:es_1} does not require boundedness of random variables $Z_i$, nor of the function $f$; 
the only crucial assumption is independence of elements in the sample $S$.
Observe that inequality~\eqref{eq:intro:es_1} basically depends on $\VKS$ and a positive free parameter $y$, which must be selected by the user. 
For instance, choosing $y = 1/n^2$ gives: 
For any $x \geq 1$, with probability at least $1-e^{-x}$ one has
\[
  |\Delta| 
  \leq \sqrt{2 (\VKS + 1/n^2) [1 + \tfrac{1}{2}\log(1 + n^2 \VKS) ] x}\,.
\]
Paraphrasing \citeauthor{kuzborskij2019efron}:
With this particular choice of $y$, the resulting inequality shows a Bernstein-type behavior, in the sense that the upper-bound is dominated by the lower-order term whenever $\VKS$ 
is small enough;
and the price for such a simple choice of $y$ is in the logarithmic term.
%

\paragraph{Remark:}
In addition to inequality~\eqref{eq:intro:es_1}, \citeauthor{kuzborskij2019efron} showed a bound that does not involve $y$ and, in particular, is scale-free:
For any $x > 0$, with probability at least $1-\sqrt{2} e^{-x}$ one has
$|\Delta| \leq 2 \sqrt{(\VKS + \E[\VKS]) x}$.

\bigskip

The remaining of this note's content is as follows.
The confidence bound of \citeauthor{kuzborskij2019efron} is presented and proved in \cref{sec:main}; and the `PAC-Bayes-ified' version of this bound is presented and proved in \cref{sec:pac_bayes}.

\section{The main result and its proof}
\label{sec:main}

\begin{theorem}
  \label{thm:main}
  Let $f : \cZ^n \to \R$ be a measurable function,
  let $\Delta = f(S) - \E[f(S)]$ be the random gap with $S$ randomly chosen from a distribution $P_n \in \cM_1(\cZ^n)$,
  and let $V = \VKS$ be the variance estimator defined in \cref{eq:VKS}. \\[1mm]
  (i) For any $x>0$, 
  \[
  \P\left(|\Delta| > 2 \sqrt{(V + \E[V]) x}\right) 
  \leq \sqrt{2} e^{-x}\,.
\]
  (ii) For any $y > 0$, and any $x \geq 1$,
  \[
  \P\left(
  |\Delta| > \sqrt{2 (V + y) [1 + \tfrac{1}{2}\log(1 + V/y) ] x}
  \right) 
  \leq e^{-x}\,.
  \]
\end{theorem}

To discuss the proof of \cref{thm:main}, the following definition will be convenient: A pair of random variables $(A,B)$ is called a \textbf{canonical pair} if $B \ge 0$ and 
\begin{align}
    \label{eq:canonical_pair}
	\sup_{\lambda\in \R} 
	\E\br{ \exp\pr{\lambda A - \frac{\lambda^2}{2} B^2 }} 
	\leq 1\,.
\end{align}
See \citet[Section 10.2]{delapena2009} for further discussion on this condition, and its connection with the so-called self-normalized processes.

A key step of the proof of \cref{thm:main} consists of establishing that $(\Delta,\sqrt{V})$ is a canonical pair. We state this as a lemma for convenient reference:
\begin{lemma}
  \label[lemma]{lem:canonical}
  $(\Delta, \sqrt{V})$ is a canonical pair.
\end{lemma}

The rest of the proof of \cref{thm:main} relies on following technical result, which essentially gives subgaussian tail probabilities for some functions of a canonical pair
(cf. \citet[Theorem 12.4 \& Corollary 12.5]{delapena2009}): 
\begin{lemma}
  \label[lemma]{lem:self_norm_concentration}
  Suppose $(A,B)$ is a canonical pair. Then: \\[1mm]
  (i) For any $t > 0$,
  \[
    \P\pr{\frac{|A|}{\sqrt{B^2 + (\E[B])^2}} \geq t} 
    \leq \sqrt{2} e^{-\frac{t^2}{4}}\,.
  \]
  (ii) For any $y > 0$ and $t \geq \sqrt{2}$,
  \[
    \P\pr{\frac{|A|}{
    \sqrt{(B^2 + y) \br{1 + \frac{1}{2} \log\pr{1 + \frac{B^2}{y}}}}
    } \geq t} 
    \leq e^{-\frac{t^2}{2}}\,.
  \]
\end{lemma}

The proof of \cref{thm:main} is then merely by combining \cref{lem:canonical} and \cref{lem:self_norm_concentration}.
Hence, 
it remains to prove \cref{lem:canonical}.
This uses the martingale method, which is at the core of the proofs of McDiarmid's/Azuma-Hoeffding's inequalities.

\begin{proof}[Proof of \cref{lem:canonical}]
Let $\E_k[\cdot]$ stand for $\E[\,\cdot\,|\, Z_1,\dots,Z_k]$.
Using the martingale difference decomposition,
the gap $\Delta = f(S) - \E[f(S)]$ can be written as
\begin{align*}
\Delta = \sum_{k=1}^n D_k
\end{align*}
where $D_k = \E_k[f(S)] - \E_{k-1}[f(S)]$. 
Notice that $D_k = \E_k[f(S)-f(S\rep{k})]$, which follows from the elementary identity 
$\E_{k-1}[f(S)] =\E_k[f(S\rep{k})]$.

The variance estimator $V = \VKS$ (cf. \cref{eq:VKS}) can be rewritten as
\[
V = \sum_{k=1}^n V_k
\]
where $V_k = \E_k\br{ \pr{f(S) - f(S\rep{k})}^2 }$.
This is just a convenient notation.

Assume for now that for every $k\in [n]$ the following holds:
\begin{align}
  \E_{k-1}\br{
    \exp\pr{ \lambda D_k - \frac{\lambda^2}{2} V_k }
    } 
  \leq 1\,.
  \label{eq:dkbkcp}
\end{align}
Then, using a recursive argument and \cref{eq:dkbkcp}, we get
\begin{align*}
  &\E\br{\exp\pr{\lambda \Delta - \frac{\lambda^2}{2} V}}
  =
   \E\br{\prod_{k=1}^n \exp\pr{
     \lambda D_k - \frac{\lambda^2}{2} V_k
     }}
  \\
  &\hspace*{7mm} =
  \E\br{ 
  	\underbrace{
	\E_{n-1}\br{
    \exp\pr{ \lambda D_n - \frac{\lambda^2}{2} V_n
    }}}_{\le 1}
    \prod_{k=1}^{n-1}
    \exp\pr{ \lambda D_k - \frac{\lambda^2}{2} V_k}
    } 
    \\
  &\hspace*{7mm} \le
  \E\br{ 
  	\underbrace{\E_{n-2}\br{
    \exp\pr{ \lambda D_{n-1} - \frac{\lambda^2}{2} V_{n-1}}
    }}_{\le 1}
    \prod_{k=1}^{n-2}
    \exp\pr{ \lambda D_k - \frac{\lambda^2}{2} V_k}
    }
     \\
    &\hspace*{7mm} \le \dots \le 1\,.
\end{align*}

Thus, it remains to prove \cref{eq:dkbkcp}.
Fix $k\in [n]$ and let $\varepsilon \in \{-1,+1\}$
be a random variable independent of $S,S'$
such that 
$\P(\varepsilon = +1) = \P(\varepsilon = -1) = 1/2$.
Let $\Delta_k = f(S) - f(S\rep{k})$.
Notice that 
$\lambda D_k - \frac{\lambda^2}{2} V_k = \E_k[\lambda \Delta_k - \frac{\lambda^2}{2} \Delta_k^2 ]$, 
and by Jensen's inequality 
\[
\exp\pr{\E_k\br{\lambda \Delta_k - \frac{\lambda^2}{2} \Delta_k^2 }}
\leq \E_k\br{\exp\pr{\lambda \Delta_k - \frac{\lambda^2}{2} \Delta_k^2}}\,.
\]
Let $\E_{-k}[\cdot]$ denote conditioning on $S$ without $Z_k$.
Then we have
\begin{align*}
\MoveEqLeft
  \E_{k-1}\br{
    \exp\pr{ \lambda D_k - \frac{\lambda^2}{2} V_k } 
    }
   \leq
  \E_{k-1}\br{
    \exp\pr{ \lambda \Delta_k - \frac{\lambda^2}{2} \Delta_k^2 }
  } \nonumber\\
  & =
    \E_{k-1}\br{
    \E_{-k}\E\br{
    \exp\pr{ \varepsilon \lambda  \Delta_k  - \frac{\lambda^2}{2} \pr{\varepsilon  \Delta_k }^2 }
    \,\Big|\, S,S' }
    }\,.
\end{align*}
The last equality follows from the assumption on the distributions, that is,
given $Z_1,\dots,Z_{k-1},Z_{k+1},\dots,Z_n$, the random variables
 $Z_k$ and $Z_k'$ are identically distributed, hence so are 
 $\Delta_k$ and $-\Delta_k$. 
 Since $x \varepsilon$ is subgaussian (for any $x \in \R$),
the innermost expectation 
in the last display is upper-bounded by one. 
%
\end{proof}

\paragraph{Remark:} 
The proof makes it clear that this inequality holds in the slightly more general setting in which $f : \cZ_1\times\cdots\times\cZ_n \to \R$ and $S = (Z_1,\ldots,Z_n)$ has independent components, where each $Z_i$ is a $\cZ_i$-valued random variable with distribution $\mu_i \in \cM_1(\cZ_i)$.

\section{PAC-Bayes-ification}
\label{sec:pac_bayes}

We adapt the notation for $\Delta$ and $V=\VKS$ to make explicit their dependence on $f$, and see them as being defined over $f$'s from some function class $\cF$:
\begin{align}
  \Delta(f) &= f(S) - \E[f(S)]\,, \tag{1'} \\
  V(f) &= \sum_{k=1}^n \E_k\left[(f(S) - f(S\rep{k}))^2 \right]\,. \tag{2'}
\end{align}
It might be convenient to make explicit the dependence of $\Delta$ and $V$ on the sample $S$ as well; 
to do so, we may write $\Delta_S(f)$ and $V_S(f)$. 
Recall that the distribution of the (size-$n$) random sample $S$ is $P_n = \mu_1 \otimes\cdots\otimes \mu_n$.
Notice that for a fixed nonrandom $s=(z_1,\ldots,z_n) \in \cZ^n$,
the gap is
\[
\Delta_s(f) = f(s) - \int_{\cZ^n} f(s') P_n(ds')\,.
\]
The expression for $V_s(f)$ is longer to write, but easy to imagine.
The point is that $\Delta$ and $V$ are real-valued functions defined over $\cZ^n \times \cF$.

Let $\cF = \{ f_\theta \}_{\theta\in\Theta}$ be a parametric family of functions $f_\theta : \cZ^n \to \R$. 
For each $\theta \in \Theta$, define
$\Delta_S(\theta)$ and $V_S(\theta)$, the gap and the variance estimator for $f_\theta(S)$. Then $(\Delta_S(\theta),\sqrt{V_S(\theta)})$ is a canonical pair, for each $\theta$, by \cref{lem:canonical}.

Given a probability kernel $Q$ from $\cZ^n$ to $\Theta$ and $s \in \cZ^n$, we write expectations with respect to the distribution $Q_s$ as $Q_s[\Delta_s] = \int_{\Theta} \Delta_s(\theta) Q_s(d\theta)$, and similarly $Q_s[V_s] = \int_{\Theta} V_s(\theta) Q_s(d\theta)$. 
If $S \sim P_n$ is the random sample, then expectations with respect to the random measure $Q_S$ are conditional expectations:
\begin{align*}
    Q_S[\Delta_S] = \E[\Delta_S(\theta) \,|\, S]\,,
    \qquad\text{and}\qquad
    Q_S[V_S] = \E[V_S(\theta) \,|\, S]\,.
\end{align*}
The joint distribution over $\cZ^n\times\Theta$ defined by $P_n$ and the probability kernel $Q$, denoted $P_n \otimes Q$, is so that choosing a random pair $(S,\theta) \sim P_n \otimes Q$ corresponds to choosing $S \sim P_n$ and then choosing $\theta \sim Q_S$. Accordingly, integrals
under $P_n \otimes Q$ correspond to the `total expectation' with respect to the random choice of $S \sim P_n$ and $\theta \sim Q_S$. For instance,
\begin{align*}
    (P_n \otimes Q)[V] 
    = \int_{\cZ^n} \int_{\Theta} V_s(\theta) Q_s(d\theta) P_n(ds)
    = \E[\E[V_S(\theta) \,|\, S]]
    = \E[V_S(\theta)]\,.
\end{align*}
With a slight abuse of notation, we may write $P_n[Q_S[V_S]]$ instead of $(P_n \otimes Q)[V]$.

\pagebreak

The `PAC-Bayes-ification' of \cref{thm:main} is as follows.

\begin{theorem}
  \label{thm:pac_bayes}
Fix an arbitrary `data-free' probability distribution $Q^0$ over $\Theta$,
and an arbitrary probability kernel $Q$ from $\cZ^n$ to $\Theta$.
%
  %
Then\\[1mm] 
  (i) For any $x > 0$,
  with probability at least $1-2 e^{-x}$ 
  we have
  \begin{equation}
    \label{eq:pac_bayes_mix_3}
    |Q_S[\Delta_S]|
    \leq
    \sqrt{
    2 (\E[V_S({\theta})] + Q_S[V_S]) \pr{\KL\pr{Q_S \,\|\, Q^0} + 2 x}
    }\,.
  \end{equation}
  (ii) For all $y > 0$ and $x \geq 1$, 
  with probability at least $1-e^{-x}$ we have
  \begin{equation}
    \label{eq:pac_bayes_mix_4}
    |Q_S[\Delta_S]|
    \leq
    \sqrt{
      2 \pr{Q_S[V_S] + y}
      \br{\KL\pr{Q_S \,\|\, Q^0} + x + \frac{x}{2} \log\pr{1 + \frac{Q_S[V_S]}{y} }}
    }\,.
  \end{equation}
\end{theorem}
The statement of this theorem uses the language of probability kernels for representing data-dependent distributions (cf. \citet{rivasplata2020}).

In the remaining of this note, we switch back to the usual notation in terms of conditional expectations: $Q_S[\Delta_S] = \E[\Delta_S(\theta) \,|\, S]$ and $Q_S[V_S] = \E[V_S(\theta) \,|\, S]$. 
Also recall that 
$\E[V_S(\theta)] = \E[\E[V_S(\theta) \,|\, S]] = P_n[Q_S[V_S]]$ 
is the total expectation.

The proof of \cref{thm:pac_bayes} is based on the following lemma.

\begin{lemma}
  \label[lemma]{lem:pac_bayes_mgf}
Under the same conditions as in \cref{thm:pac_bayes}. \\[1mm] 
(i) For all $x \geq 0$, 
\begin{equation}
  \label{eq:pac_bayes_mix_2}
  \E\br{
    \exp\cbr{
      x \sqrt{
      \left(
      \frac{\E[\Delta_S(\theta) \,|\, S]^2}{\E[V_S(\theta)] + \E[V_S(\theta) \,|\, S]} - 2 \KL\pr{Q_S \,\Vert\, Q^0}
      \right)_+
      }
    }
  }
  \leq 2 e^{x^2}\,.
\end{equation}
(ii) For any $y > 0$, we have
  \begin{equation}
    \label{eq:pac_bayes_mix_1}
    \E\br{\frac{y}{\sqrt{y^2 + \E[V_S(\theta) \,|\, S]}} \, \exp\cbr{\frac{\E[\Delta_S(\theta) \,|\, S]^2}{2 (y^2 + \E[V_S(\theta) \,|\, S])} - \KL\pr{Q_S \,\Vert\, Q^0}}}
  \leq
  1\,.
  \end{equation}
\end{lemma}

\begin{proof}[Proof of \cref{lem:pac_bayes_mgf}]
For convenience, we start with the proof of \cref{eq:pac_bayes_mix_1}.
Recall the following change of measure, which is the basis of the PAC-Bayesian analysis:
Let $\pi$ and $\pi^0$ be probability measures on $\Theta$, and let the induced expectation operators be $\E$ and $\E^0$, respectively.
Let $X$ be a $\Theta$-valued random variable.
Then, for any measurable function $f \,:\, \Theta \to \R$ we have
    \[
      \E[f(X)] \leq \KL(\pi\,\Vert\,\pi^0) + \log \E^0\br{e^{f(X)}}\,.
    \]
Below we use this with $\pi = Q_S$, $\pi^0 = Q^0$, 
and $f_S(\theta) = \lambda \Delta_S(\theta) - \frac{\lambda^2}{2} V_S(\theta)$. 

Let $\E$ and $\E^0$ be the expectation 
with respect to $P_n \otimes Q$ and $P_n \otimes Q^0$, respectively.
Conditioning on the random sample $S$ we have:
\begin{align*}
  \E\br{ \lambda \Delta_S(\theta) - \frac{\lambda^2}{2} V_S(\theta) \,\middle|\, S}
  &\leq
    \KL\pr{Q_S \,\Vert\, Q^0}
    +
    \log \E^0\br{ e^{\lambda \Delta_S(\theta) - \frac{\lambda^2}{2} V_S(\theta)} \,\middle|\, S}\,.
\end{align*}
Subtracting the KL term, and taking exponential on both sides gives
\begin{align*}
  e^{\E\br{ \lambda \Delta_S(\theta) - \frac{\lambda^2}{2} V_S(\theta) \,\middle|\, S} - \KL\pr{Q_S \,\Vert\, Q^0}}
  &\leq
    \E^0\br{ e^{\lambda \Delta_S(\theta) - \frac{\lambda^2}{2} V_S(\theta)} \,\middle|\, S}\,.
\end{align*}
Then, taking expectation over the random sample $S$ on both sides, and keeping in mind that $(\Delta_S(\theta), \sqrt{V_S(\theta)})$ is a canonical pair for any fixed $\theta$, we have
\begin{align*}
  \E\br{e^{\E\br{ \lambda \Delta_S(\theta) - \frac{\lambda^2}{2} V_S(\theta) \,\middle|\, S} - \KL\pr{Q_S \,\Vert\, Q^0}}}
  &\leq
    \E^0\br{ e^{\lambda \Delta_S(\theta) - \frac{\lambda^2}{2} V_S(\theta)}} \\
  &=
    \E^0\br{ \E^0\br{ e^{\lambda \Delta_S(\theta) - \frac{\lambda^2}{2} V_S(\theta)} \, \middle| \, \theta } }
  \leq
    1\,,
\end{align*}
The equality is by swapping the order of expectation, which is possible since $Q^0$ is a data-free distribution (cf. \cite{rivasplata2020}).
Next, multiplying both sides by $e^{-\lambda^2 y^2/2} y / \sqrt{2 \pi}$ for some fixed $y > 0$, integrating with respect to $\lambda \in \R$, and applying Fubini's theorem, gives%
\footnote{
This is inspired by the proof of \cite[Theorem 12.4]{delapena2009}, which uses the method of mixtures with a Gaussian distribution. 
}
\begin{align*}  
  \E\br{
  e^{- \KL\pr{Q_S \,\Vert\, Q^0}}
  \frac{y}{\sqrt{2 \pi}} \int_{-\infty}^{\infty} e^{
  \lambda \E\br{\Delta_S(\theta) \,|\, S } - \frac{\lambda^2}{2} \E\br{V_S(\theta) \,\middle|\, S} - \frac{\lambda^2}{2} y^2
  } \diff \lambda
  }
  \leq 1\,.
\end{align*}
Carrying out the Gaussian integration we arrive at
\begin{equation*}
  \E\br{\frac{y}{\sqrt{y^2 + \E[V_S(\theta) \,|\, S]}} \, \exp\cbr{\frac{\E[\Delta_S(\theta) \,|\, S]^2}{2 (y^2 + \E[V_S(\theta) \,|\, S])} - \KL\pr{Q_S \,\Vert\, Q^0}}}
  \leq
  1\,,
\end{equation*}
which finishes the proof of~\cref{eq:pac_bayes_mix_1}.

For the other part of the lemma, we consider the following:

\begin{claim}
\label[claim]{lem:subG_from_exp}
Let $U$ be a non-negative random variable, and for $\alpha>0$ define $C(\alpha) =\E\br{\exp(\alpha U^2)}$. Then, for any $x\ge 0$,
$\E\br{\exp(x U )} \le C(\alpha) e^{x^2/4\alpha}$.
\end{claim}

The proof of this claim is as follows.
Fix $\alpha>0$ and $x \geq 0$.
Using the inequality $ab \leq (a^2 + b^2)/2$ with $a = x/\sqrt{2 \alpha}$ and $b=\sqrt{2\alpha}U$ we have
\begin{align*}
x U = 
\frac{x}{\sqrt{2\alpha}} \sqrt{2\alpha}U \le \frac{x^2}{4\alpha} + \alpha U^2\,.
\end{align*}
Then take exponential on both sides, and take expectations.

Next, we see the proof of part (i) of the lemma.

Consider the random variable
\[
U=\sqrt{
\left(\frac{\E[\Delta_S(\theta) \,|\, S]^2}{\E[V_S(\theta)] + \E[V_S(\theta) \,|\, S]} - 2 \KL\pr{Q_S \,\Vert\, Q^0}\right)_+
}
\]
%
and notice that \cref{eq:pac_bayes_mix_2} follows from the claim with $\alpha = 1/4$, provided that we show that $C(1/4) = \E\br{\exp\pr{U^2/4}} \le 2$.
For this, consider an arbitrary $y>0$,
and consider the abbreviations 
%
%
\[
A=\frac{y}{\sqrt{y^2 + \E[V_S(\theta) \,|\, S]}}\,,
\hspace{9mm}
B=\frac{\E[\Delta_S(\theta) \,|\, S]^2}{2 (y^2 + \E[V_S(\theta) \,|\, S])} - \KL\pr{Q_S \,\Vert\, Q^0}\,.
\]
We need to upper-bound $\E\br{\exp\pr{U^2/4}} = \E\br{\exp(\pr{B}_+/2)}$.
Keeping in mind that $A>0$ (in fact, $A \in (0,1]$),
by Cauchy-Schwarz,
\begin{align*}
\E\br{\exp(\pr{B}_+/2)}
  &=\E\br{\exp\pr{\pr{B}_+/2} A^{1/2} A^{-1/2} } \nonumber \\
  &\le
  \sqrt{\E\br{A \exp\pr{\pr{B}_+} }} \, \sqrt{ \E\br{ A^{-1} } }  \,.
\end{align*}
Observe that $\E[A \exp(B)] \leq 1$ by~\cref{eq:pac_bayes_mix_1},
and $A \in (0,1]$.
Now, we have
\begin{align*}
  &\sqrt{\E\br{A \exp\pr{(B)_+}} \E\br{\frac{1}{A}}} \\
  &\hspace*{5mm}=
    \sqrt{\Big(
    \E\br{A \, \mathds{1}\{B \geq 0\} \exp\pr{B}} + \E\br{A \, \mathds{1}\{B < 0\}}
    \Big)
    \E\br{\frac{1}{A}}
    }
  \leq
    \sqrt{\E\br{\frac{2}{A}}}\,.
\end{align*}
Finally, by subadditivity of the square root function and Jensen's inequality,
\[
  \sqrt{\E\br{\frac{2}{A}}}
  =
    \sqrt{2 \E\br{\sqrt{\frac{y^2 + \E[V_S(\theta) \,|\, S]}{y^2}}}}
  \leq
    \sqrt{2 + 2 \frac{\sqrt{\E[V_S(\theta)]}}{y}}
  \leq
    2\,,
\]
where the last inequality is by taking any $y \geq \sqrt{\E[V_S(\theta)]}$.
Thus, $C(1/4) \leq 2$ for the chosen $U$.
Applying \cref{lem:subG_from_exp} with $\alpha = 1/4$ completes the proof.
\end{proof}

To complete the argument, the proof of \cref{thm:pac_bayes} is given next.

\begin{proof}[Proof of \cref{thm:pac_bayes}]
Applying Chernoff's bounding technique with~\cref{eq:pac_bayes_mix_2} gives
\begin{align*}
  \P\pr{
  \sqrt{\pr{\frac{\E[\Delta_S(\theta) \,|\, S]^2}{\E[V_S(\theta)] + \E[V_S(\theta) \,|\, S]} - 2 \KL\pr{Q_S \,\Vert\, Q^0}}_+}
  \geq t
  }
  &\leq
  2 \inf_{x \geq 0} e^{x^2 - t x}\,.
\end{align*}
The infimum is $e^{-t^2/4}$.
Thus, with probability at least $1-2 e^{-x}$ we have
\begin{align*}
  \pr{\frac{\E[\Delta_S(\theta) \,|\, S]^2}{\E[V_S(\theta)] + \E[V_S(\theta) \,|\, S]} - 2 \KL\pr{Q_S \,\Vert\, Q^0}}_+
  \leq
  4 x\,.
\end{align*}
With some algebra, this event implies
\begin{align*}
  |\E[\Delta_S(\theta) \,|\, S]|
  \leq
  \sqrt{(\E[V_S(\theta)] + \E[V_S(\theta) \,|\, S]) \pr{2 \KL\pr{Q_S \,\Vert\, Q^0} + 4 x}}\,.
\end{align*}
The last display is equivalent to~\cref{eq:pac_bayes_mix_3}.
Hence \cref{thm:pac_bayes}(i) is proved.

Next, observe that for any $y > 0$ and any $t \geq \sqrt{2}$,
\begin{align*}
\MoveEqLeft
  \P\pr{ \frac{\E[\Delta_S(\theta) \,|\, S]^2}{2 (y^2 + \E[V_S(\theta) \,|\, S])} - \KL\pr{Q_S \,\Vert\, Q^0} \geq \frac{t^2}{2} \br{1 + \frac{1}{2} \log\pr{1 + \frac{\E[V_S(\theta) \,|\, S]}{y^2}}} }\\
  &\hspace*{-5mm}\leq
    \P\pr{ \frac{\E[\Delta_S(\theta) \,|\, S]^2}{2 (y^2 + \E[V_S(\theta) \,|\, S])} - \KL\pr{Q_S \,\Vert\, Q^0} \geq \frac{t^2}{2} + \frac{1}{2} \log\pr{1 + \frac{\E[V_S(\theta) \,|\, S]}{y^2}} }\\
  &\hspace*{-5mm}=
    \P\pr{ \frac{\E[\Delta_S(\theta) \,|\, S]^2}{2 (y^2 + \E[V_S(\theta) \,|\, S])} - \KL\pr{Q_S \,\Vert\, Q^0} - \frac{1}{2} \log\pr{1 + \frac{\E[V_S(\theta) \,|\, S]}{y^2}} \geq \frac{t^2}{2} }\\
  &\hspace*{-5mm}\leq
    \E\br{\sqrt{\frac{y^2}{\E[V_S(\theta) \,|\, S] + y^2}} \, \exp\cbr{\frac{\E[\Delta_S(\theta) \,|\, S]^2}{2 (y^2 + \E[V_S(\theta) \,|\, S])} - \KL(Q_S\,\Vert\,Q^0)} } \, e^{-\frac{t^2}{2}}\\
  &\hspace*{-5mm}\leq
    e^{-\frac{t^2}{2}}\,,
\end{align*}
where the last two inequalities follow from Markov's inequality and~\cref{eq:pac_bayes_mix_1}.
This implies that for all $x \geq 1$, 
with probability at least $1-e^{-x}$, one has
\begin{align*}
  \frac{\E[\Delta_S(\theta) \,|\, S]^2}{2 (y^2 + \E[V_S(\theta) \,|\, S])} - \KL\pr{Q_S \,\Vert\, Q^0}
  \leq
  x\pr{1 + \frac{1}{2} \log\pr{1 + \frac{\E[V_S(\theta) \,|\, S]}{y^2}}}\,.
\end{align*}
Notice that $y^2$ may be replaced with $y$, since $y>0$ is a free variable.
Doing this replacement, and rearranging the terms, 
we get the equivalent of~\cref{eq:pac_bayes_mix_4}.
Hence \cref{thm:pac_bayes}(ii) is proved.
\end{proof}

\paragraph{Closing remarks.}
\citeauthor{kuzborskij2019efron} deserve fair credit for showing that the pair $(\Delta,\sqrt{V})$ meets \citet{delapena2009}'s `canonical condition' (\cref{lem:canonical}), which enabled powerful tools for bounding exponential moments.
Of course, this was possible with their variance estimator $V = \VKS$.
Apart from that, the main part of the work of \citeauthor{kuzborskij2019efron} is in the proofs of \cref{lem:pac_bayes_mgf} and \cref{thm:pac_bayes}, which cleverly use the techniques of \citeauthor{delapena2009}.
In the next iteration of this note (provided that enough readers cared about it) I intend to add discussions about \cref{thm:main} \& \cref{thm:pac_bayes}, and applications.

\bibliography{biblia}
\nocite{*} 

\end{document}